\documentclass{amsart}
\usepackage{amsmath,amsfonts, verbatim}
\usepackage{tikz,pgf}

\usepackage{euscript}
\usepackage{graphicx}
\usepackage{color}
\usepackage{url}
\usepackage{verbatim}

\newcommand{\TN}{\EuScript{T}_n}
\newcommand{\A}{\simeq_A}

\newtheorem{theorem}{Theorem}[section]
\newtheorem*{theorem*}{Theorem}

\newtheorem{proposition}[theorem]{Proposition}
\newtheorem{lemma}[theorem]{Lemma}
\newtheorem{corollary}[theorem]{Corollary}
\newtheorem*{corollary*}{Corollary}

\theoremstyle{definition}
\newtheorem{definition}[theorem]{Definition}

\theoremstyle{remark}

\newtheorem{remark}[theorem]{Remark}

\begin{document}

\markboth{H\'{e}l\`{e}ne Barcelo, Christopher Severs, Jacob A. White}
{Discrete Fundamental Group of Associahedron and Exchange Module}

\title{The Discrete Fundamental Group of the Associahedron, and the Exchange Module}

\author{H\'{e}l\`{e}ne Barcelo }
\address{Mathematical Sciences Research Institute, 17 Gauss Way\\
 Berkeley, CA 94720, United States \\
hbarcelo@msri.org}

\author{Christopher Severs }
\address{eBay Incorporated,  2065 Hamilton Avenue \\
San Jose, CA 95125, United States \\
chris.severs@gmail.com}

\author{Jacob A. White}
\address{Texas A\&M University, Department of Mathematics, Mailstop 3368 \\
College Station, TX 77843-3368, United States \\
jwhite@math.tamu.edu}

\maketitle

\begin{abstract}
The associahedron is an object that has been well studied and has numerous applications, particularly in the theory of operads, the study of non-crossing partitions, 
lattice theory and more recently in the study of cluster algebras. We approach the associahedron from the point of view of discrete homotopy theory. We study the abelianization of the discrete fundamental group, 
and show that it is free abelian of rank $\binom{n+2}{4}$. We also find a combinatorial description for a basis of this rank.
We also introduce the exchange module of the type $A_n$ cluster algebra, used to model the relations in the cluster algebra.
We use the discrete fundamental group to the study of exchange module, and show that it is also free abelian of rank $\binom{n+2}{3}$.


\end{abstract}


\section{Introduction}

The associahedron is a polytope first discovered by Stasheff in the course of his research on operads \cite{stasheff-h} and later rediscovered by Haiman (unpublished) and Lee \cite{lee}. 
The $1$-skeleton of the associahedron has also appeared under a different name, the Tamari lattice \cite{tamari}. Since the associahedron is a simple polytope, it is dual to a simplicial polytope. 
The boundary of the dual polytope is known as the cluster complex of type $A_n$, due to its connections with the cluster algebra of type $A$.
The vertex set consists of all diagonals of a regular $(n+3)$-gon and the facets correspond to triangulations of 
the same $(n+3)$-gon. We denote this complex $\TN$. 
Thus, the vertices of the associahedron are triangulations, 
and there is an edge between two vertices if the corresponding triangulations differ by a single diagonal flip. An example of the 3 dimensional associahedron appears in Fig. \ref{fig:3dasc}.
One goal of this paper is to study the discrete fundamental group of the associahedron.

The notion of discrete homotopy theory has been around for almost a decade. Its origin can be traced back to the early 1970 with the work of Atkin \cite{atkin1}. 
A few years later, Maurer \cite{maurer1} constructed a new fundamental group for graphs, that turns out to be an extension of Atkin's work. This theory was rediscovered later by Malle \cite{malle}.
Finally, in 2001 Barcelo et al \cite{bklw} defined a notion of ``discrete" homotopy theory for graphs and for any simplicial complex. Their work turned out to be a generalization of the work of Atkins, Maurer, and Malle.
The notion of discrete homotopy was developed to capture some combinatorial properties of simplicial complexes (including graphs) that the classical homotopy theory could not account for.
Since then, quite surprisingly, the notion of discrete fundamental groups has turned up in several areas of mathematics .
More strikingly, in a recent paper, Barcelo and Smith \cite{barcelo-smith} showed that the discrete fundamental group of the permutahedron is isomorphic to the (classical) fundamental group of the complement (in $\mathbb{R}^n$) 
of a subspace arrangement called the 3-equal arrangement, a generalisation of the braid arrangement. Furthermore, these results generalize to (real) Coxeter groups of all types \cite{barcelo-severs-white}.

Inspired by these results we computed the discrete fundamental group of other polyhedra, and in particular, that of the associahedron.
 Our main result is the following:
\begin{theorem}
\label{thm:intro2}
$A_1^{n-2}(\TN)^{ab}$ is a free abelian group with a minimal generating set of size $\binom{n+2}{4}$, where $A_1^{n-2}(\TN)^{ab}$ is the abelianization of the discrete fundamental group of $\TN$.
\end{theorem}
We actually derive a combinatorial description of the generating set mentioned in Theorem \ref{thm:intro2}, in terms of $5$-sets $S$ of $[n+3]$ such that $1 \in S$.

Encouraged by the fact that the discrete fundamental group of the Permutohedron is the (classical) fundamental group of a subspace arrangement, 
we sought to find a relationship between the discrete fundamental group of the associahedron and the type $A$ cluster algebra. The motivation for studying cluster algebras comes from the fact that the type $A$ cluster algebra and the 
associahedron are intrinsically related. We discovered that the discrete fundamental group of the associahedron arises in the study of an abelian group which we call the \emph{exchange module} $E(\mathcal{A}_n)$. 
While we do not believe that $E(\mathcal{A}_n)$ is known yet in the literature, we are hopeful that it may find application in cluster algebras, or one the other various places associahedra arise.

A cluster algebra, introduced by Fomin and Zelevinksy in \cite{fomin-zel-cluster}, is an axiomatically defined commutative ring equipped with a subset of generators, and a process for creating the rest of the generators via 
\emph{mutation}.
Though developed in the context of representation theory, cluster algebras have found applications in
discrete dynamical systems based on rational recurrences \cite{fomin-zel-laurent}, $Y$-systems in thermodynamic Bethe Ansatz \cite{fomin-zel-y}, grassmannians 
and their tropical analogues \cite{scott-grass}, Poisson geometry and Teichm\"uller theory \cite{fock-goncharov}. For more information about cluster algebras 
we refer the reader to the Park City notes by Fomin and Reading \cite{fomin-reading}, the CDM '03 notes by Fomin and Zelevinsky \cite{fomin-zel-cdm}, 
the original series of cluster algebra papers by Fomin and Zelevinsky \cite{fomin-zel-cluster,fomin-zel-finite,fomin-zel-y} and of course the comprehensive 
Cluster Algebras Portal maintained by Fomin at \url{http://www.math.lsa.umich.edu/~fomin/cluster.html}.

We are interested in understanding the relationships among clusters for cluster algebras of type A. One way of doing so is to study the
(discrete) fundamental group of the exchange graph of the associahedron; loops represent ways of obtaining the same cluster after a sequence of mutations. 
Since mutations involving disjoint quadrilaterals are easy to understand, it makes sense to add 2-cells to the exchange graph. As a result, the discrete fundamental group models 
the relationships among the exchange relations, modulo the obvious commutativity relations.

We are able to describe the abelianization of the discrete fundamental group for the associahedron in terms of 
the cluster algebra of type $A$. The resulting abelian group we refer to as the exchange module $E(\mathcal{A}_n)$. 
We are able to 
prove the following theorem:
\begin{theorem}
$E(\mathcal{A}_n)$ is free abelian of rank $\binom{n+2}{3}$. 

\end{theorem}

\begin{figure}[htbp]

\begin{center}
\scalebox{.5}{\includegraphics{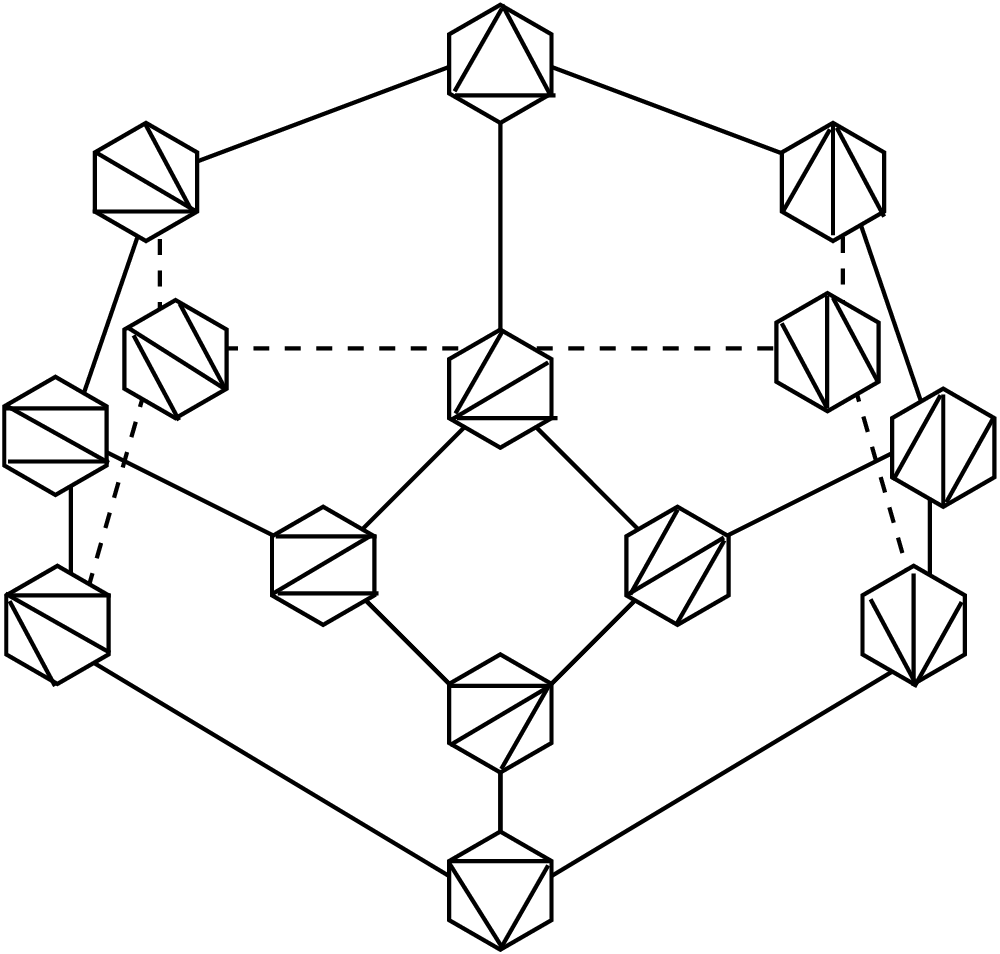}}

\caption{The 3-dimensional associahedron, with vertices labeled by triangulations.}
\label{fig:3dasc}
\end{center}
\end{figure}

We start with a review of discrete homotopy theory. Next, we study the discrete homotopy group of $\TN$ in Sections \ref{sec:labels} and \ref{sec:maintheorem}. 
In Section \ref{sec:maintheorem} we prove Theorem \ref{thm:intro2}.  In Section \ref{sec:review}, we review terminology from cluster algebras and define the exchange module. The material in that section does not depend on the rest of the paper; 
the reader may feel free to read Section \ref{sec:review} first to learn the definition of the exchange module. In Section \ref{sec:cluster}, we 
study the exchange module in more detail, using our results about the discrete fundamental group of the associahedron.
We conclude with some remarks about applying our methods to various generalizations of the associahedron.

\section{Discrete Homotopy Theory} \label{sec:atheory}

All of the definitions and theorems in this section are from \cite{bklw}. 

Discrete homotopy theory is a tool to study the combinatorial connectedness of a simplicial complex $\Delta$. 
Given an integer parameter $0 \le q \le dim(\Delta)-1$, then two simplices $\sigma$ and $\tau$ are \emph{$q$-near} if they share a $q$-face. 
A \emph{$q$-chain} is a sequence \[ (\sigma_0, \sigma_1, \ldots, \sigma_k) \] of simplices such that $\sigma_i$ and $\sigma_{i+1}$ are $q$-near 
for $ \le i \le k-1$. Given a fixed simple $\sigma_0$, a $q$-chain that starts and ends with $\sigma_0$ is a \emph{$q$-loop based at $\sigma_0$}. 
Two simplices are \emph{$q$-connected} if there is a $q$-chain between them. 

\begin{definition}
 \label{def:ahomotopic}
Define an equivalence relation $\A$ on the set of all $q$-loops with a common base simplex $\sigma_0$ in the following way:
\begin{enumerate}

 \item  
$(\sigma_0, \sigma_1, \ldots, \sigma_k, \sigma_0) 
 \A (\sigma_0, \sigma_1, \ldots, \sigma_k, \sigma_0, \sigma_0)$

\item $(\sigma) \A (\tau)$ if they are of the same length and there is a grid between them as in Fig. \ref{fig:grid}, where each row of the grid is a $q$-loop, and each column is a $q$-chain.
\begin{figure}[htbp]

\begin{center}
\scalebox{.5}{\input{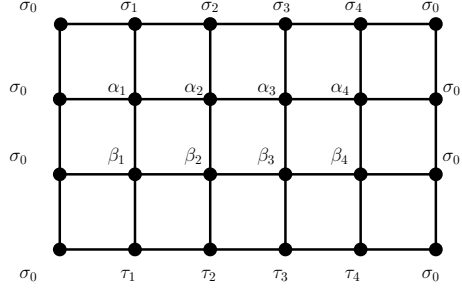}}
\caption{A grid between $(\sigma)$ and $(\tau)$. Edges indicate simplices are $q$-near.}
\label{fig:grid}
\end{center}
\end{figure}

\end{enumerate}

\end{definition}

\begin{theorem}[Proposition 2.3 in \cite{bklw}]
\label{thm:group}
 The collection of all equivalence classes of $q$-loops in $\Delta$ based at $\sigma_0$ forms a group, denoted $A_1^q(\Delta,\sigma_0)$. The group operation induced by concatenation of loops. 
The identity is the equivalence class of constant loop $\sigma_0$.  The inverse of $[\sigma]$ is the class $[\sigma^{-1}]$, where $\sigma^{-1}$ is the loop $\sigma$ read in reverse order. 
 \end{theorem}

The group $A_1^q(\Delta,\sigma_0)$ is the $q$-discrete fundamental group of $\Delta$. We shall also call it the discrete fundamental group when $q$ is fixed. There are higher discrete 
homotopy groups as well; we refer the interested reader to \cite{bklw} for more information about them. 

In practice we suppress the base point from the $q$-discrete fundamental group notation, since Proposition 2.4 of \cite{bklw} states that if a complex $\Delta$ is 
$q$-connected and $\sigma_0$ and $\tau_0$ are maximal simplices, then $A_1^q(\Delta,\sigma_0) \cong A_1^q(\Delta,\tau_0)$.

One way to understand $A_1^q(\Delta,\sigma_0)$ is to consider an associated object called the connectivity graph. 

\begin{definition}
 \label{def:gammagraph}
The connectivity graph $\Gamma^q_{max}(\Delta)$ is the graph whose nodes are the maximal simplices of $\Delta$. 
There is an edge between node $\sigma$ and node $\tau$ if the simplices $\sigma$ and $\tau$ are $q$-near. 
\end{definition}

Closed walks in $\Gamma^q_{max}(\Delta)$ based at $\sigma_0$ may be viewed as $q$-loops in $\Delta$ based at $\sigma_0$. It turns out that every $q$-loop is discrete homotopy equivalent to a loop corresponding to a closed walk in $\Gamma^q_{max}(\Delta)$.
 We will use loops and closed walks interchangeably because of this correspondence. 

The graph $\Gamma^q_{max}(\Delta)$ provides a link between the $q$-discrete fundamental group of $\Delta$ and the fundamental group, $\pi_1$, of a related topological space.

\begin{theorem}[\cite{bklw}, Theorem 5.16]
 \label{thm:gammacomplex}
$A_1^q(\Delta,\sigma_0) \simeq \pi_1(X_{\Gamma^q_{max}(\Delta)}, \sigma_0)$, where $X_{\Gamma^q_{max}(\Delta)}$ is the regular 2-cell complex obtained by attaching a 
2-cell along the boundary of every 3- and 4-cycle of the graph $\Gamma^q_{max}(\Delta)$, and we use $\sigma_0$ to denote both a simplex in $\Delta$ and the corresponding node in $\Gamma$.
\end{theorem}

Thus, two loops $(\sigma)$ and $(\tau)$ are homotopic if and only if their associated walks in $\Gamma^q_{max}(\Delta)$ differ by 3- and 4-cycles only. In fact, a homotopy grid can be constructed by stretching the walk $(\sigma)$ around individual 3- and 4-cycles to obtain the sequence of loops that make up the rows of the grid. 

In this paper we are concerned with the abelianization of the discrete fundamental group. From the isomorphism in Theorem \ref{thm:gammacomplex}, it follows that, if
$X_{\Gamma^q_{max}(\Delta)}$ is connected,  then the abelianization $A_1^q(\TN)^{ab}$ is isomorphic to the first homology group of $X_{\Gamma^q_{max}(\Delta)}$, $H_1(X_{\Gamma^q_{max}(\Delta)})$.

While this review contains information about discrete homotopy theory that we need, it is not comprehensive. We refer the interested reader to \cite{bbdl} for more details and a category theoretic treatment of this material. 

\section{Discrete Homotopy Classes of $\TN$} \label{sec:labels}

The object we apply discrete homotopy theory to is the cluster complex $\TN$. The vertices of $\TN$ are diagonals in a regular polygon $P$ on $n+3$ vertices. Simplices of $\TN$ are sets of pairwise non crossing diagonals. Hence, maximal simplices in this complex correspond to triangulations of $P$. Thus,  $\Gamma=\Gamma^{n-2}_{max}(\TN)$ is the graph whose nodes are triangulations of $P_{n+3}$ and whose edges 
corresponding to diagonal flips. This graph is the $1$-skeleton of the \emph{associahedron}.  

For our paper, cycles of a graph $\Gamma$ come equipped with one of two possible orientations. That is, a cycle based at $v_0$ is a sequence $C = (v_0, \ldots, v_k)$ of vertices, such that $v_i$ is adjacent to $v_{i+1}$ for all $i$, $v_k = v_0$, and $v_1, \ldots, v_{k-1}$ are all distinct. In particular, $-C = (v_k, v_{k-1}, \ldots, v_1, v_0)$ is regarded as a different cycle.

Consider the cycles in $\Gamma$ that bound $2$-cells of the associahedron. We use the convention of Fomin, Shapiro and Thurston in 
\cite{fomin-shapiro-thurston} and refer to these cycles as \emph{geodesic cycles}. Recall that $2$-cells of the associahedron correspond to codimension 2 simplices of $\TN$, which give a polygonal dissection of $P_{n+3}$ into $n-2$ 
triangles and one pentagon, or $n-3$ triangles and $2$ quadrilaterals. 

\begin{proposition}
\label{prop:4and5cycles}
Let $T \in \TN$ be of codimension 2. If $T$ leaves two quadrilaterals inside $P_{n+3}$ untriangulated then $T$ corresponds to a geodesic 4-cycle. 
Otherwise, $T$ corresponds to a geodesic 5-cycle.  
\end{proposition}

\begin{proof}
Suppose that two quadrilaterals are left untriangulated. There are four ways to triangulate these regions, resulting in the four nodes of $\Gamma$ that bound a geodesic 4-cycle. 

If there are no two untriangulated quadrilaterals, then there is one untriangulated pentagon.  There are five ways to triangulate this region, resulting in five nodes of $\Gamma$ bounding a geodesic 5-cycle. 
  \end{proof}

Another useful fact is the following, due to Fomin et al.
\begin{proposition}[Theorem 3.10 in \cite{fomin-shapiro-thurston}]
Fix a base point $v_0 \in \Gamma$.
 $\pi_1(\Gamma)$ is generated by all based loops of the form $PCP^{-1}$, where $P$ is a path from $v_0$ to some $v_i$, and $C$ is a geodesic cycle.
\end{proposition}
By Theorem \ref{thm:gammacomplex}, there is a 2-dimensional cell complex $X_{\Gamma}$ such that $A_1^{n-2}(\TN) \cong \pi_1(X_{\Gamma})$. Moreover, $X_{\Gamma}$ is 
constructed by attaching 2-cells to every geodesic 4-cycles of $\Gamma$. Thus, loops of the form $PCP^{-1}$, where $C$ is a $4$-cycle, are contractible in $X_{\Gamma}$. 
\begin{corollary}
\label{cor:a1}
Fix a maximal simplex $\sigma_0 \in \TN$.
 Then $A_1^{n-2}(\TN)$ is generated by homotopy classes $[\sigma \tau \sigma^{-1}]$, where $\sigma$ is a $(n-2)$-chain starting at $\sigma_0$, and $\tau$ is a $(n-2)$-loop corresponding to a geodesic 5-cycle in $\Gamma$.
\end{corollary}
To understand $A_1^{n-2}(\TN)$ and its abelianization, it is enough to understand equivalence classes coming from geodesic $5$-cycles.

The goal of the remainder of this section is to define an edge-labeling of $\Gamma$, 
and use the labeling to construct a homomorphism $\tilde{\psi}: A_1^{n-2}(\TN) \to \Lambda$, where $\Lambda$ is a free abelian group with basis given by the edge labels. 
Now we proceed to define an edge-labeling of $\Gamma$.
\begin{definition}
\label{def:edgelabel}
 Let $E$ be an edge in $\Gamma$ with corresponding diagonal flip occurring in a quadrilateral (of $P_{n+3}$) whose boundary vertices are $a,b,c,d$. 
We define the label of $E$, $L(E)$ to be the set $\{a,b,c,d \}$. Let $\Lambda$ be the free $\mathbb{Z}$-module with basis given by all sets $\{a,b,c,d \} \subset [n+3]$.
\end{definition}

The same label will be applied to many edges in $\Gamma$. 
An important fact about the edge labels of geodesic 4-cycles is that opposite edges have the same label.

In order to define $\tilde{\psi}$, we need to orient the edges of $\Gamma$.
Suppose we obtain a triangulation $T'$ from a triangulation $T$ by removing a diagonal $\alpha$ and replacing it with a diagonal $\beta$. Suppose $\alpha$ has endpoints $a < b$ and $\beta$ has endpoints 
$c < d$. If $a < c$, then define $\epsilon(T, T') = 1$. Otherwise, let $\epsilon(T, T') = -1$.

\begin{definition}
 Fix a base node $T_0$ in $\Gamma$. Let $(T)=(T_0, T_1, \ldots, T_k, T_0)$ be an $(n-2)$-loop (hereafter just a loop). 
\begin{enumerate}
 \item Let \begin{equation} \psi((T)) = \sum_{i=0}^k \epsilon(T_i, T_{i+1}) L(T_i T_{i+1})\label{eq:definepsi} \end{equation} where $\epsilon(T_i, T_{i+1}) = 0$ whenever $T_i = T_{i+1}$.
\item Given a homotopy class $[T]$, with representative loop $(T)$, let $\tilde{\psi}([T]) = \psi((T))$.
\end{enumerate}

\end{definition}
We show that $\tilde{\psi}: A_1^{n-2}(\TN) \to \Lambda$ is a well-defined homomorphism. As a first step, we show that, for any loop 
$(\sigma) = (\sigma_0, \sigma_1, \sigma_2, \sigma_3, \sigma_4, \sigma_0)$, we have $\psi((\sigma)) = 0$, where $\sigma_0$ is not required to be the base simplex. 
For simplicity, we shall only discuss two examples of such loops $(\sigma)$.
One interesting case is a situation of the form $(\sigma_0, \sigma_0, \sigma_1, \sigma_2, \sigma_0)$. Then $(\sigma_0, \sigma_1, \sigma_2)$ is a 3-cycle 
in $\Gamma$. Since $\Gamma$ is triangle-free, this case does not occur. The other interesting case is when the loop is of the form $(\pi) = (\pi, \rho, \sigma, \tau, \pi)$, where all simplices are 
distinct. This corresponds to a 4-cycle in $\Gamma$. Moreover, if $\rho$ is obtained from $\pi$ by replacing a diagonal $\alpha$ with a diagonal $\beta$, then $\tau$ is obtained from $\sigma$ by replacing 
$\beta$ with $\alpha$. In particular, these two edges have the same label, and opposite sign. Similarly, $(\rho, \sigma)$ and $(\tau, \pi)$ get the same label, but opposite sign. Hence, even in 
this case, $\psi((\pi)) = 0$.

\begin{proposition}
Let $\tilde{\psi}: A_1^{n-2}(\TN) \to \Lambda$ be given by $\tilde{\psi}([\sigma]) = \psi((\sigma))$ for any representative $(\sigma)$. Then $\tilde{\psi}$ is a well-defined homomorphism.
\end{proposition}
\begin{proof}
 We show the map $\tilde{\psi}$ is well-defined. 
 
 Consider two loops $(\sigma), (\tau)$ such that $(\sigma) \simeq_A (\tau)$. 
Let $(\sigma) = (\sigma_0, \ldots, \sigma_i, \sigma_0)$, and $(\tau) = (\tau_0, \ldots, \tau_r, \tau_0)$. 
Let $(\sigma') = (\sigma_0, \ldots, \sigma_i, \sigma_0, \sigma_0)$. Then $\psi((\sigma')) = \psi((\sigma))$. Thus, without loss of generality, we 
may assume $(\sigma)$ and $(\tau)$ are of the same length $r$, and there exists a discrete homotopy grid from $(\sigma)$ to $(\tau)$. 

As a first step, assume that $\sigma_i$ is $(n-2)$-near to $\tau_i$ for all $i$.
\begin{displaymath}
 \begin{array}{ccc}
  \psi((\rho)) - \psi((\tau)) & = &  \sum_{i=0}^r (\epsilon(\sigma_i, \sigma_{i+1}) L(\sigma_i \sigma_{i+1}) - \epsilon(\tau_i, \tau_{i+1}) L(\tau_i \tau_{i+1})) \\
& = & \sum_{i=0}^r (\epsilon(\sigma_i,\sigma_{i+1}) L(\sigma_i \sigma_{i+1}) + \epsilon(\sigma_{i+1}, \tau_{i+1}) L(\sigma_{i+1} \tau_{i+1}) \\
&  & -\epsilon(\tau_i, \tau_{i+1}) L(\tau_i \tau_{i+1}) - \epsilon(\sigma_i, \tau_i) L(\sigma_i \tau_i)) \\
& = & \sum_{i=0}^r \psi((\pi)^i) \\
& = & 0
\end{array}
\end{displaymath} 
where the last equality comes from the fact that $\psi((\sigma)) = 0$ for any loop of length four.

Finally, suppose $(\sigma)$ and $(\tau)$ are the same length, and consider a discrete homotopy grid from $(\sigma)$ to $(\tau)$. Let $(\rho)^0, \ldots, (\rho)^m$ 
denote the loops formed by the rows of the homotopy grid. Then by our previous argument, $\psi((\rho)^0) = \psi((\rho)^1) = \cdots = \psi((\rho)^m)$.  

Finally, it is not hard to show that the map $\tilde{\psi}$ is a homomorphism.
  \end{proof}

\section{A Calculation of $A_1^{n-2}(\TN)^{ab}$.}\label{sec:maintheorem}

In this section, we focus on $A_1^{n-2}(\TN)^{ab}$. By Corollary \ref{cor:a1} in the last section, $A_1^{n-2}(\TN)$ is generated by all homotopy classes of the form 
$[\sigma \tau \sigma^{-1}]$, where $\sigma$ was a $(n-2)$-chain starting at the fixed base point, and $\tau$ was a $(n-2)$-loop corresponding to a geodesic $5$-cycle. 
Fix a base simplex $T_0$. Given two 5-cycles $C$ and $C'$, we say $C \simeq C'$ if there exists paths $P$ and $Q$ such that $PCP^{-1}$, $QC'Q^{-1}$ are loops, and $PCP^{-1} \simeq_A QC'Q^{-1}$.
Let $\rho: A_1^{n-2}(\TN) \to A_1^{n-2}(\TN)^{ab}$ be the projection map. Then $\rho([PCP^{-1}]) = \rho([QC'Q^{-1}])$. Then $A_1^{n-2}(\TN)^{ab}$ is generated by all equivalence classes of geodesic 5-cycles, under the equivalence relation $\simeq$. 
Our first goal is to understand when two geodesic 5-cycles are representatives of the same element of $A_1^{n-2}(\TN)^{ab}$. To this end, we introduce labels and 
orientations for geodesic 5-cycles.

\begin{definition}
\label{def:cyclelabel}
 Let $C$ be a geodesic 5-cycle in $\Gamma$ with an untriangulated region (of $P_{n+3}$) bounded by $a,b,c,d,e$. The label of $C$, $L(C)$ is the set $\{a,b,c,d,e \}$. 
\end{definition}

\begin{remark}
 \label{rem:edgeunion}
Let $C$ be a geodesic 5-cycle in $\Gamma$ with edges $E_1, \ldots, E_5$. Then:
\begin{enumerate}
\item $L(C) = \bigcup_{i=1}^5 L(E_i)$. 

\item $L(E_i) \neq L(E_j)$ for $i \neq j$. 
\end{enumerate}
\end{remark}

Let $C$ be a geodesic 5-cycle, with label set $\{a,b,c,d,e \}$, with $a < b < c < d < e$. Let $T_i, T_{i+1}$ be the vertices of $C$ such that $L(T_i T_{i+1}) = \{b,c,d,e \}$.
Then define $\epsilon(C) = \epsilon(T_i, T_{i+1})$. This defines an orientation function from all geodesic 5-cycles to $\{+, - \}$.

Distinct cycles may share the same label. In fact, we can build a net of 4-cycles between any two geodesic 5-cycles with the same label by using a series of diagonal flips which correspond to a 
path between the two cycles. We say two geodesic $5$-cycles $C$ and $C'$ differ by a net of $4$-cycles if there is a sequence $C_1, \ldots, C_k$ of geodesic $5$-cycles, 
where $C_i$ has vertices $v_{1,i}, \ldots, v_{5,i}$ such that:
\begin{enumerate}
\item $C_1 = C$, $C_k = C'$
\item For each $v_{i,j}$ is adjacent to $v_{i, j+1}$ for $1 \leq i \leq 5$, $1 \leq j \leq k-1$
\end{enumerate}

\begin{proposition}
 \label{prop:edgepath}
Let $C$ and $C'$ be geodesic 5-cycles in $\Gamma$ with $L(C) = L(C')$ and $\epsilon(C) = \epsilon(C')$. Then there is a net of $4$-cycles between $C$ and $C'$.
\end{proposition}

\begin{proof}
Let $L(C) = \{a,b,c,d,e \}$ with $1 \leq a < b < c < d < e \leq n+3$.
The polygonal dissection $D$ of $P_{n+3}$ corresponding to the cycle $C$ has $n-1$ triangles, and one pentagon $P'$. The pentagon $P'$ has vertices $a,b,c,d,e$. 
Let $R_{ab}$ be the polygon with vertices $a, a+1, \ldots, b$. Suppose $R_{ab}$ has at least three vertices. $D$ induces a triangulation $T$ on $R_{ab}$. 
Similarly, $C'$ corresponds to a polygonal dissection $D'$, which induces a triangulation $T'$ on $R_{ab}$. Suppose that $T \neq T'$. 
There is a sequence of diagonal flips $f_1, \ldots, f_k$ in $R$ which lead from $T$ to $T'$. Since the diagonal flips occur in $R_{ab}$, 
they do not intersect the interior of $P'$. 

Let $T_0, T_1, T_2, T_3, T_4$ be the nodes of $C$ such that $T_i$ is adjacent to $T_{i+1}$ (where addition is modulo 5). For $0 \leq i \leq 4$, $1 \leq j \leq k$, 
define $T_{i,j} = f_j f_{j-1} \cdots f_1 T_i$. Then $T_{0,k}, T_{1,k}, T_{2,k}, T_{3,k}, T_{4,k}$ forms the node set of a geodesic 5-cycle $C^*$. 
Moreover, the corresponding polygonal dissection $D^*$ consists of $(D \cup T') \setminus T$. 
Also, $\{T_{i,j}: 0 \leq i \leq 4, 1 \leq j \leq k \}$ is the vertex set of a net of $4$-cycles from $C$ to $C'$. This is because for 
each $i$ and $j$, $T_{i,j+1}$ differs from $T_{i,j}$ by a diagonal flip $f_j$. Moreover, the flip corresponding to the edge $T_{i,j}, T_{i+1,j}$ 
involves diagonals in $P'$, while the flip $f_j$ involves only diagonas in $R_{ab}$. Thus, the flips commute, and so $T_{i, j+1}$ is adjacent to $T_{i+1, j+1}$.

Let $R_{bc}$ be the polygon with vertices $b, b+1, \ldots, c$. By the same argument, we can construct a net of $4$-cycles between $C'$ and a new 
5-cycle whose corresponding polygonal dissection agrees with $D$ outside of $R_{ab}$ and $R_{bc}$, and agrees with $D'$ in the regions $R_{ab}, R_{bc}$. 
Continuing in this fashion, one eventually constructs a net of $4$-cycles from $C$ to $C'$.   \end{proof}

In the last section we defined a homomorphism $\psi:A_1^{n-2}(\TN) \to \Lambda$, where $\Lambda$ is the free abelian group generated 
by all edge labels. Since it is abelian, there is an induced map $\psi:A_1^{n-2}(\TN)^{ab} \to \Lambda$, which we use in the proof of the following theorem.
\begin{theorem}
 \label{thm:maintheorem}
Let $C$ and $C'$ be geodesic 5-cycles. Then $C \A C'$ if and only if $L(C)=L(C')$ and $\epsilon(C) = \epsilon(C')$. 
\end{theorem}
\begin{proof}
First, suppose that $L(C) \neq L(C')$. Then, in particular, there is an edge $e$ of $C$ such that $L(e) \neq L(e')$ for all edges $e'$ on $C'$. 
Note then that the coefficient of $L(e)$ in $\psi(C)$ is non-zero, yet the coefficient of $L(e)$ in $\psi(C')$ is zero. Then $\psi(C) \neq \psi(C')$, and hence $C$ is not equivalent to $C'$.
So suppose that $L(C) = L(C') = \{a,b,c,d,e \}$, with $a < b < c < d < e$, and suppose further that $\epsilon(C) = +$ and $\epsilon(C') = -$.
 Let $e$ be the edge of $C$ such that $L(e) = \{b,c,d,e \}$. Then the coefficient of $L(e)$ in $\psi(C)$ is 1, and yet the coefficient of $L(e)$ in $\psi(C')$ is $-1$. 
Again we conclude that $C$ is not equivalent to $C'$.

Finally, suppose that $L(C) = L(C')$ and $\epsilon(C) = \epsilon(C')$.
Then by Proposition \ref{prop:edgepath}, there is a net of 4-cycles between $C$ and $C'$, and hence they are equivalent.   \end{proof}
As a result of Theorem \ref{thm:maintheorem}, we can now associate a label to any equivalence class $[C]$ of geodesic 5-cycles. That is, 
if $C$ is a representative 5-cycle in $[C]$, we define $L([C]) = L(C)$. Of course, under this notation, $L(-[C]) = L([C])$. Finally, given a label set $S$, 
let $[C]_S$ represent the equivalence class of $5$-cycles such that $L(C) = S$ and $\epsilon(C) = +$.

The set $\{[C]_S: S \subset [n+3], |S| = 5 \}$ forms a generating set of size $\binom{n+3}{5}$ for $A_1^{n-2}(\TN)^{ab}$. However, $A_1^{n-2}(\TN)^{ab}$ is free abelian on a subset of this generating set.
We first present an example. Consider the graph $\Gamma^{1}(\EuScript{T}_3)$, shown in Fig. \ref{fig:3dasclabels}. 

\begin{figure}[htbp]
\begin{center}
\scalebox{.5}{\includegraphics{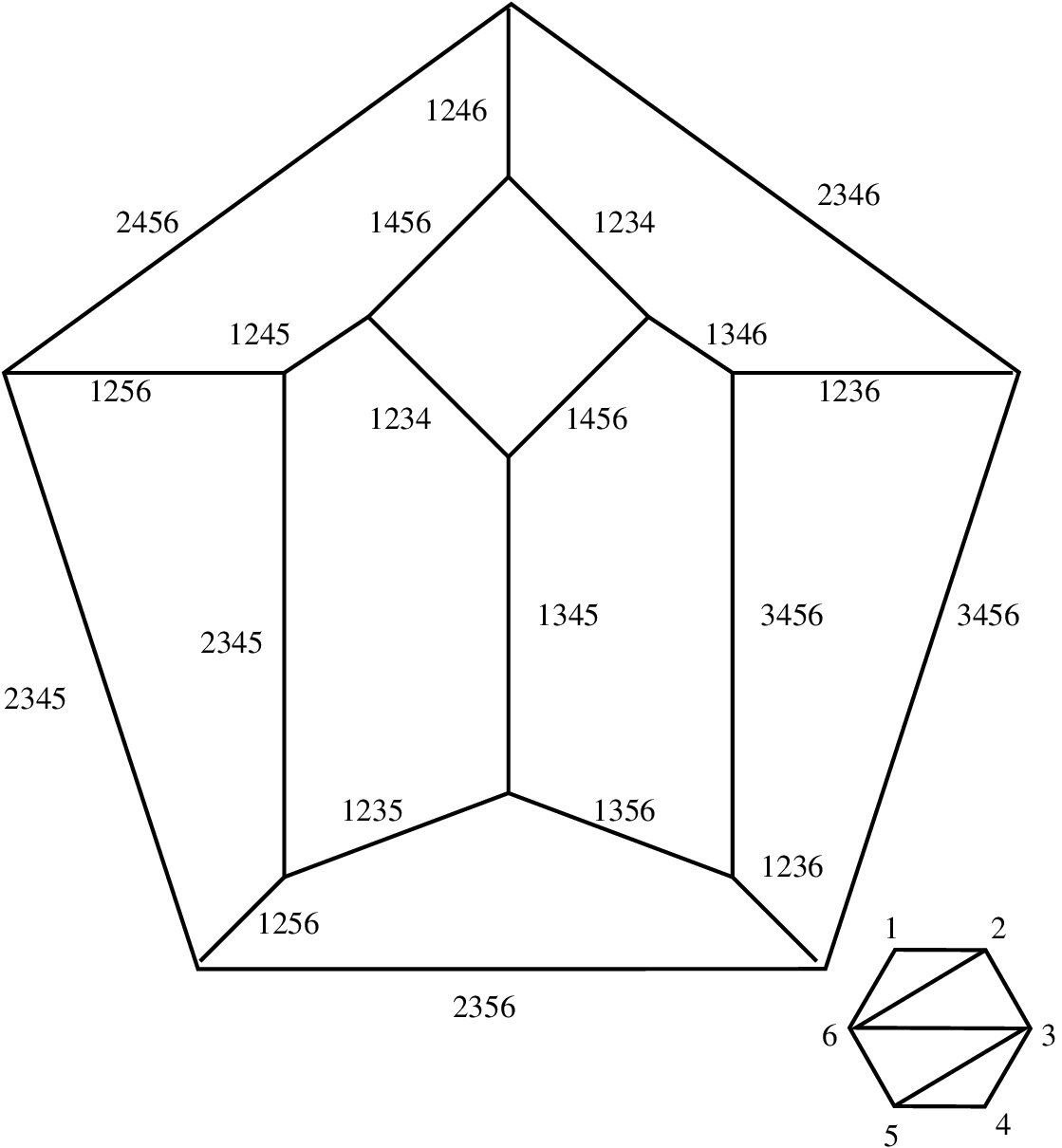}}

\caption{The graph $\Gamma^{1}(\EuScript{T}_3)$ with edge labels. One of the lower right vertices is labeled with the corresponding triangulation of the hexagon.}
\label{fig:3dasclabels}

\end{center}
\end{figure}

The geodesic 5-cycle with label $\{2,3,4,5,6\}$ can be written as a sum of geodesic 4-cycles and the geodesic 5-cycles that have a 1 in their label. 
Thus, we can $A_1^1(\EuScript{T}_3)^{ab}$ actually has rank at most 5. Our main theorem claims that the rank, in this case, is 5:
\begin{theorem}
\label{thm:basisforabelian}
 $A_1^{n-2}(\TN)^{ab}$ is free abelian of rank $\binom{n+2}{4}$. A basis is given by $B = \{[C]_S: S \subset [n+3], |S| = 5, 1 \in S \}$, 
where $[C]_S$ is the equivalence class of geodesic 5-cycles whose label set is $S$, and such that $\epsilon(C) = +$.
\end{theorem}
Note that Theorem \ref{thm:intro2} follows from this theorem.
\begin{proof}
This proof contains two parts. First, we show that, for any geodesic 5-cycle $C$ without a 1 in its label, there is another geodesic 
$5$-cycle $C'$, such that $C' \A C$ and $C'$ may be written as a sum of geodesic 5-cycles with a 1 in their labels. Then we show that there are no 
relations between geodesic 5-cycles that have distinct labels which contain the element 1. 

Let $C$ be a geodesic 5-cycle with $L(C)=\{a,b,c,d,e\}$, $a < b < c < d < e$ and such that $1 \notin L(C)$. Let $D$ be the corresponding polygonal dissection of 
$P_{n+3}$ into $n-1$ triangles and one pentagon, whose vertices are $a,b,c,d,e$. Suppose that one of the triangles of $D$ has vertices $1,a,e$. 
Let $\Gamma'$ be the subgraph of $\Gamma$ induced by the vertex set $\{T: D \setminus \{ae \} \subset T \}$. Then  $\Gamma'$ is isomorphic to the graph in Fig. \ref{fig:3dasclabels}, where $2,3,4,5,6$ are replaced by $a,b,c,d,e$. 
The geodesic 5-cycle on the boundary of 
the graph is $C$, and is equal to a sum of those cycles with a 1 in their label. 

Suppose then that $D$ does not have a triangle with vertices $1,a$ and $e$. It suffices to show that there exists a cycle $C'$, such that $C \A C'$, 
and the polygonal dissection corresponding to $C'$ does have a triangle with vertices $1,a$ and $e$.

Consider the polygon $P'$ given by vertices $1, 2, \ldots, a-1, a, e, e+1, \ldots, n+3, 1$. $D$ induces a triangulation $T$ on $P'$. There exists a triangulation 
$T'$ of $P'$ that has a triangle with vertices $1,a$ and $e$. Let $D'$ be polygonal dissection $(D \cup T') \setminus T$, and let $C'$ be the corresponding 
5-cycle of $\Gamma$. Since $L(C') = \{a,b,c,d,e \} = L(C)$, we must have  $C \A C'$, by Theorem \ref{thm:maintheorem}. Since $D'$ does have a triangle with vertices 
$1, a$, and $e$, it is the sum of 5-cycles that have $1$ in their label sets. Hence every equivalence class of $5$-cycles $[C]$ can be written as a linear combination of elements of $B$.

 Given $x \in \Lambda$, we write $x = \sum_{1 \leq a < b < c < d \leq n+3} \langle x, \{a,b,c,d \} \rangle$. Then, for all $S \in B$, $1 < a < b < c < d \leq n+3$, 
we have $\langle \psi([C]_S), \{a,b,c,d\} \rangle = \delta_{S \setminus \{1\}, \{a,b,c,d \}}$, the Kronecker Delta function.
We wish to show that $B$ is linearly independent over $\mathbb{Z}$. 
Suppose we have $x = \sum_{S \in B} c_S [C]_S = 0$. 
Then $\psi(x) = \sum_{S \in B} c_S \psi([C]_S) = 0$. Then for every $1 < a < b < c < d$, we have 
\begin{displaymath} 
\begin{array}{ccc} 
 \langle \psi(x), \{a,b,c,d\} \rangle & = & \sum_{S \in B} c_S \langle \varphi([C]_S), \{a,b,c,d \} \rangle \\ & = & \sum_{S \in B} c_S \delta_{S \setminus \{1\}, \{a,b,c,d\}} \\ & = & c_{\{1,a,b,c,d\}}
\end{array}
\end{displaymath}
However, since $\langle \psi(x), \{a,b,c,d \} \rangle = 0$, it follows that $c_{\{1,a,b,c,d \}} = 0$. In particular, $x = 0$. Thus, $B$ is linearly independent, 
and forms a minimal generating set for $A_1^{n-2}(\TN)^{ab}$.   \end{proof}

\section{The Type $A$ Cluster Algebra and the Exchange Module} \label{sec:review}
Now that we understand the structure of the abelianization of the discrete fundamental group, it would be nice to have a cluster algebra description of this group. 
That is, since the discrete fundamental group represents relations in the mutation process, modulo commuting mutations, 
it should be possible to describe the abelianization in terms of relations in the mutation process. To this end, we define the exchange module, 
and show that it is isomorphic to the abelianization of $A_1^{n-2}(\TN)$.

First we review the construction of the type $A_n$ cluster algebra. First, let $P_n$ be a polygon on $n+3$ vertices. Fix a triangulation $T$ of $P$; label the diagonals of $T$ with the numbers 
$1, \ldots, n$ in some order, and the boundary edges of $P$ with $n+1, \ldots, 2n+3$. 
Now consider indeterminates $x_1, \ldots, x_{2n+3}$, where $x_1, \ldots, x_n$ (called \emph{cluster variables}), and $x_{n+1}, \ldots, x_{2n+3}$ (\emph{frozen variables}).
The cluster algebra of type $A$ is a subring of $\mathbb{Q}(x_1, \ldots, x_{2n+3})$, with one generator $x_{k'}$ for each diagonal $k$, and subject to relations of the form \begin{equation} x_k x_{k'}= x_ax_c + x_b x_d,\label{eq:cluster_relation} \end{equation}
 where $a,b,c,d$ are diagonals bounding a quadrilateral, and the diagonals $k$ and $k'$ intersect inside the quadrilateral. 
One can prove that each generator is in fact a Laurent polynomial in the original variables. The generators $x_{k'}$ are also referred to as cluster variables.
For general cluster algebras, the notion of diagonal flip is replaced with the notion of mutation, the concept of triangulation is replaced with the concept of cluster, 
and the initial triangulation $T$ is replaced with an initial \emph{seed}. However, we choose to only give the definitions for type $A$, since this cluster algebra is the focus of the current paper.

We now define the exchange module, one of the groups that is the focus of this paper.
Every set of noncrossing diagonals induces a polygonal dissection of $P$. Suppose $D$ is a set of noncrossing diagonals for which the polygonal dissection 
consists of triangles and exactly one pentagon. Label the vertices of $P$ clockwise as $1, \ldots, n+3$, and let $a,b,c,d,e$ be the vertices of the pentagon, with 
$a < b < c < d < e$. Given all ways to complete $D$ to get a triangulation $T$, we obtain a cluster. Then we have one mutation relation 
for each diagonal flip that occurs inside of the pentagon. We represent these relations as variables $x_{ac,bd}$, where we have flipped the diagonal with endpoints $a$ and $c$ to the diagonal with endpoints $b$ and $d$. 
From considering diagonal flips involving diagonals in the interior of the pentagon, we define \begin{equation} X_{a,b,c,d,e} =  
x_{ac, bd} + x_{ad, be} + x_{bd, ce} - x_{be, ac} - x_{ce, ad} \label{eq:pentagonalvariable} \end{equation}
to be the \emph{pentagonal relation} for the embedded pentagon with endpoints $a,b,c,d,e$. The diagonals appearing in the relation depend only on the boundary of pentagon.
These relations are \emph{pentagonal relations}. Note that these correspond to relations among clusters. That is, given a cluster corresponding to $T$, if we flip the diagonals of the pentagon in a certain order, the resulting cluster is the 
same as the starting cluster. This pentagonal relation phenomenon is interesting, as it does not seem to correspond to a classically studied group. 
The exchange module, defined below, is an attempt to capture this phenomenon as 

\begin{definition}
Fix a polygon $P = P_{n+3}$ with $n+3$ vertices, labeled $1, \ldots, n+3$ in clockwise order. Let $Cr(P)$ denote the set $\{ \{\alpha, \beta \}: 
\alpha \mbox{ and } \beta \mbox{ is a pair of crossing } \newline \mbox{ diagonals} \}$. We often refer to a diagonal $\alpha$ by its endpoints. 
That is, $13$ represents the diagonal with endpoints $1$ and $3$. Given $1 \leq a < b < c < d < e \leq n+3$, let 
$X_{\{a,b,c,d,e\}}$ be the pentagonal relation defined in Eq. \eqref{eq:pentagonalvariable}.
 The \textbf{exchange module}, $E(\mathcal{A}_n)$, is the abelian group with generating set $\{x_{\{\alpha, \beta\}}: \{\alpha, \beta\} \in Cr(P) \}$, and subject to all relations of the form
$X_{\{a,b,c,d,e\}} = 0$, for all $1 \leq a < b < c < d < e \leq n+3$. 

\label{def:primary}
\end{definition}
Note that we will abuse notation, and let $x_{\alpha, \beta} = x_{\{ \alpha, \beta \}}$.

Assume that we have ordered the diagonals in the following way: for diagonals $\alpha$ and $\beta$, $\alpha < \beta$ if 
$\min \alpha < \min \beta$, where we identify diagonals with their endpoints.
The sequence of diagonal flips in Fig. \ref{fig:pentagonal} gives the linear relation $x_{\{13, 24\}} + x_{\{14, 25\}} + x_{\{24, 35\}} - x_{\{25, 13\}} - x_{\{35, 14\}} = X_{\{1,2,3,4,5\}} = 0$.
Positive signs in the linear relation correspond to cases when flipping replaces a diagonal $\alpha$ with a larger diagonal $\beta$.
 This is the motivation behind the defining relations for $E(\mathcal{A}_n)$.

The exchange module is generated by variables corresponding to diagonal flips, and has relations corresponding to pentagons. Thus, the exchange module is defined entirely in terms of the combinatorics of the type $A$ cluster algebra, 
and the study of recurrence relations which appear in the exchange process. It seems like $E(\mathcal{A}_n)$ loses information, since we are no longer keeping track of the clusters. In the next section, we shall see that 
$E(\mathcal{A}_n) \simeq A_1^{n-3}(\TN)^{ab}$, and hence the information lost corresponds to abelianization. Thus, in some sense the exchange module is an abelian `approximation' to the study of the mutation process for type $A$ 
cluster algebras.

\section{Rank of the Exchange Module} \label{sec:cluster}

The goal of this section is the prove the following theorem.
\begin{theorem}
\label{thm:isomorphism0}
$E(\mathcal{A}_n)$ is free abelian, of rank $\binom{n+2}{3}$. 

\end{theorem}

Before we continue, let us fix some notation.
 Given $1 \leq a < b < c < d < e \leq n+3$, let $X_{a,b,c,d,e}$ be given by Eq. \ref{eq:pentagonalvariable}.
Let $F(\mathcal{A}_n)$ be the free $\mathbb{Z}$-module with generating set $\{ y_{\{\alpha, \beta\}}: \{\alpha, \beta \} \in Cr(P) \}$. 
 Let $\theta: F(\mathcal{A}_n) \to E(\mathcal{A}_n)$ be given by \begin{equation} \theta(y_{\{\alpha, \beta \}}) = x_{\{\alpha, \beta\}} \label{eq:definetheta} \end{equation}

We prove Theorem \ref{thm:isomorphism0} by showing that $\ker \theta$ is isomorphic to $A_1^{n-2}(\TN)^{ab}$. Since we know $A_1^{n-2}(\TN)^{ab}$ is free abelian, we get a basis for 
$\ker \theta$, which we use to show that $E(\mathcal{A}_n)$ is isomorphic to a subgroup of $F(\mathcal{A}_n)$ of rank $\binom{n+2}{3}$.

 \begin{figure}[htbp]
\begin{center}
\scalebox{.15}{\includegraphics{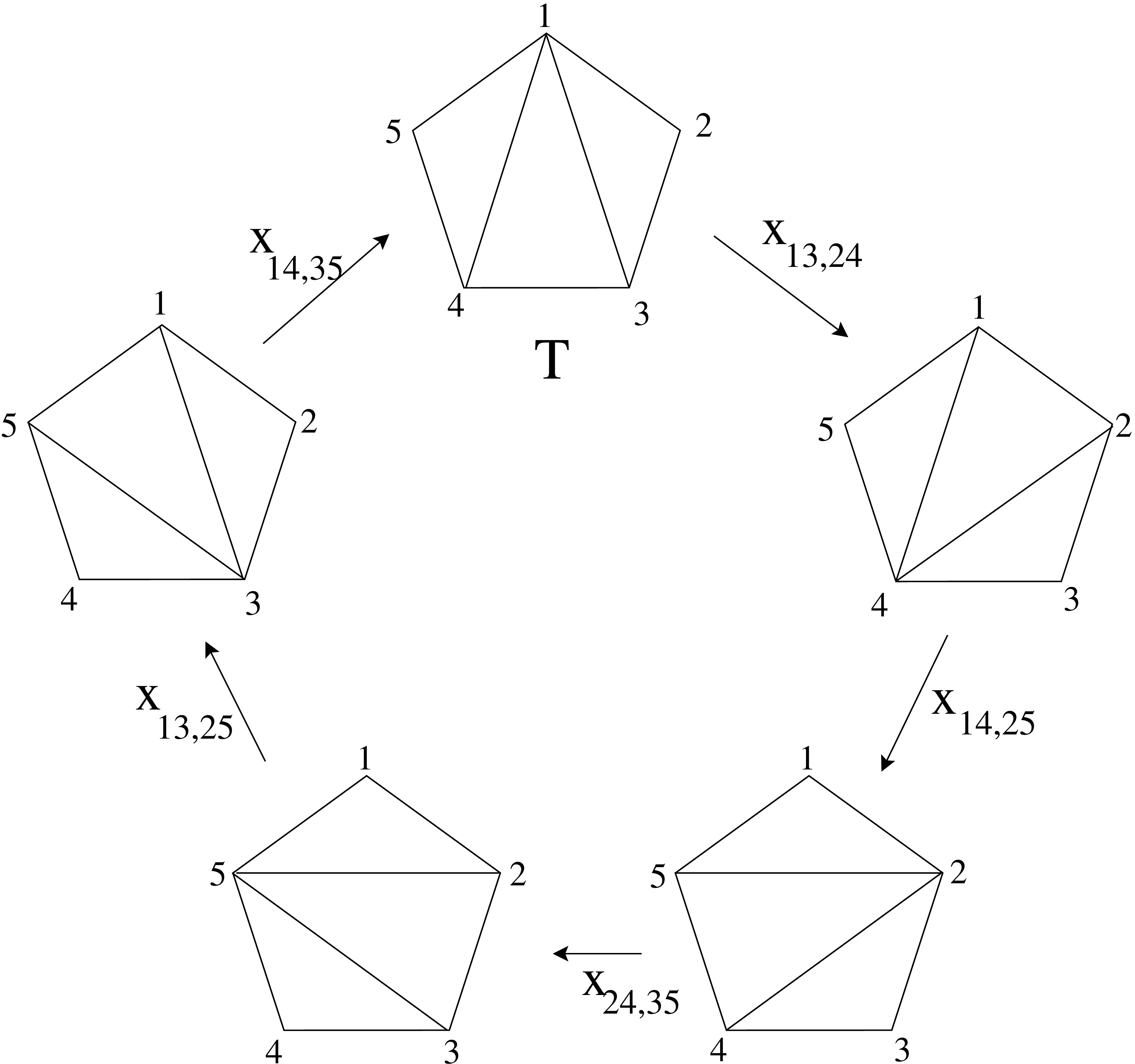}}
\caption{A sequence of flips corresponding to a pentagonal recurrence.}
\label{fig:pentagonal}
\end{center}
\end{figure}

Let $$L = \{S: S \subseteq [n+3], |S| = 5, 1 \in S \}.$$ By Theorem \ref{thm:basisforabelian}, $A_1^{n-2}(\TN)^{ab}$ is free abelian with basis $\{[C]_S: S \in L \}$, 
where $[C]_S$ is the equivalence class of geodesic 5-cycles with label set $S$ and positive sign. We show that $\ker \theta$ is free abelian, and $\{Y_S: S \in L \}$ is a basis for $\ker \theta$. 
For $S \in L$, let \begin{equation} \varphi([C]_S) = Y_S \label{eq:defvarphi} \end{equation} and extend the map by linearity, to get a homomorphism $\varphi: A_1^{n-2}(\TN)^{ab} \to \ker \theta$.
It suffices to show that $\varphi$ is an isomorphism.

For the proof, it is useful to fix some notation. 
Consider a free $\mathbb{Z}$-module $M$ with basis $B$. Given $x \in M$, we let $x = \sum_{b \in B} \langle x,b \rangle b$. 
That is, when we express $x$ in terms of the basis $B$, $\langle x, b \rangle$ denotes the coefficient of 
$b$.

\begin{theorem}
 \label{thm:isomorphism1}
The map $\varphi: A_1^{n-2}(\TN)^{ab} \to \ker \theta$ is an isomorphism.
\end{theorem}
\begin{proof}
$\ker \theta$ is generated by $\{Y_S: S \subset [n+3], |S| = 5 \}$. 
Consider the set $B' = \{Y_S: S \in L \}$.
It suffices to show that, given any $1 < a < b < c < d < e \leq n+3$, $Y_{a,b,c,d,e}$ can be expressed as a linear combination of elements of $B'$.
Observe that:  
\begin{displaymath}
\begin{array}{lll}
 Y_{a,b,c,d,e} & = & y_{ac, bd} + y_{ad,be} + y_{bd, ce} - y_{ac, be} - y_{ad, ce}  \\
               & = & y_{ac, bd} - y_{1b,ad} - y_{1c, bd} + y_{1b, ac} + y_{1c, ad}  \\
               & + & y_{ad, be} - y_{1b,ae} - y_{1d, be} + y_{1b, ad} + y_{1d, ae}  \\
               & + & y_{bd, ce} - y_{1c,be} - y_{1d, ce} + y_{1c, bd} + y_{1d, be}  \\
               & - & y_{ac, be} - y_{1c,ae} - y_{1b, ac} + y_{1c, be} + y_{1b, ae}  \\
               & - & y_{ad, ce} - y_{1d,ae} - y_{1c, ad} + y_{1d, ce} + y_{1c, ae}  \\
               & = & Y_{1,a,b,c,d} + Y_{1,a,b,d,e} + Y_{1,b,c,d,e} - Y_{1,a,b,c,e} - Y_{1,a,c,d,e}
\end{array}
\end{displaymath}
Therefore, the map $\varphi$ is surjective. 

Now we show that $\ker \varphi$ is trivial. Let $[C']$ be such that $\varphi([C']) = 0$. Fix $1 < a < c < b < d \leq n+3$. 
Since $\varphi([C']) \equiv 0$, we must have $\langle \varphi([C']), y_{ac,bd} \rangle = 0$. 
By Lemma \ref{lem:coeffs}, it follows that $\langle [C'], [C]_{1,a,b,c,d} \rangle = 0$. 
Since $\{[C]_S: S \in L \}$ forms a basis for $A_1^{n-2}(\TN)^{ab}$, we see that $[C'] = 0$ in $A_1^{n-2}(\TN)^{ab}$.
  \end{proof}

\begin{lemma}
 \label{lem:coeffs}
Let $1 < a < b < c < d \leq n+3$, and consider $[C] \in A_1^{n-2}(\TN)^{ab}$. 
Then $\langle [C],[C]_{1,a,b,c,d} \rangle = \langle \varphi([C]), y_{ac,bd} \rangle$.
\end{lemma}
\begin{proof}
Note that, for any $S \subset [n+3]$ with $1 \in S$ and $|S| = 5$, we have $\langle Y_S, y_{ac,bd} \rangle = \delta_{S, \{1,a,b,c,d \}}$, the 
Kronecker delta function. 
Given $[C] \in A_1^{n-2}(\TN)^{ab}$, we see that 
\begin{displaymath}
 \begin{array}{ccc}
  \langle \varphi([C]), y_{ac,bd} \rangle & = & \sum_{S \in L} \langle [C], [C]_S \rangle \langle \varphi([C]_S), y_{ac,bd} \rangle \\ 
& = & \sum_{S \in L} \langle [C], [C]_S \rangle \langle Y_S, y_{ac,bd} \rangle \\
& = & \sum_{S \in L} \langle [C], [C]_S \rangle \delta_{S, \{1,a,b,c,d \}} \\
& = & \langle [C], [C]_{1,a,b,c,d} \rangle
 \end{array}
\end{displaymath}   \end{proof}

\begin{corollary}
\label{cor:maincor}
 $E(\mathcal{A}_n) \simeq F(\mathcal{A}_n) / \ker \theta$. In particular, the exchange module is a quotient of a free module by $A_1^{n-2}(\TN)^{ab}$.
\end{corollary}

\begin{remark}
 \label{rem:transferbasis}
 We have $\ker \theta$ is free abelian of rank $\binom{n+2}{4}$, 
and is freely generated by $\{Y_{1,a,b,c,d}: 1 < a < b < c < d \leq n+3 \}$. This follows from the proof of Theorem \ref{thm:basisforabelian} and the above isomorphism $\varphi$.
\end{remark}

Let $G$ be the free abelian subgroup of $F(\mathcal{A}_n)$ generated by \\ $B^* := \{y_{\alpha, \beta}: 1 \mbox{ is an endpoint of } \alpha \}$. 

\begin{theorem}
 The map $\theta|_G : G \to E(\mathcal{A}_n)$ is an isomorphism.
\end{theorem}

\begin{proof}
 $E(\mathcal{A}_n)$ is generated by $\{x_{ac,bd}: 1 < a < b < c < d \leq n+3 \}$. It suffices to show that, given $1 < a < b < c < d \leq n+3$, $x_{ac,bd}$ can be expressed as a sum of elements from $\theta(B^*)$. 
The expression $X_{1,a,b,c,d} = 0 $ can be rewritten as
\begin{equation} x_{ac,bd} = - x_{1b, ac} - x_{1c,ad} + x_{1b, ad} + x_{1c, bd}. \label{eq:pentagonalvariabletwo} \end{equation}
All the terms on the right hand side are in $\theta(B^*)$, so $\theta|_G$ is surjective.

We show that $\ker \theta|_G$ is trivial. Let $y \in G$, such that $\theta|_G(y) = 0$. Then $y \in \ker \theta$, which has basis $B' = \{Y_S: S \subset [n+3], |S| = 5, 1 \in S \}$.
However, for any choice $1 < a < b < c < d \leq n+3$, we see that
\begin{displaymath} 
\begin{array}{ccc} 
 \langle y, Y_{1,a,b,c,d}\rangle & = & \langle \varphi^{-1}(y), [C]_{1,a,b,c,d} \rangle \\ & = & \langle y, y_{ab,cd} \rangle \\ & = & 0
\end{array}
\end{displaymath}
where the first equality comes from the fact that $\varphi$ is an isomorphism, the second equality comes from Lemma \ref{lem:coeffs}, 
and the third equality comes from the fact that $y \in G$ and 
$y_{ab,cd} \not\in G$. So $\langle y,Y_{1,a,b,c,d} \rangle = 0$ for all $1 < a < b < c < d \leq n+3$. However, by Remark \ref{rem:transferbasis}, the set $\{Y_{1,a,b,c,d}: 1 < a < b < c < d \leq n+3 \}$ forms a 
basis for $\ker \theta$. 
We conclude that $y \equiv 0$, and hence $\ker \theta|_G$ is trivial.
  \end{proof}
Since $G$ is free abelian of rank $\binom{n+2}{3}$, so is $E(\mathcal{A}_n)$. Moreover, we obtain an explicit basis, given by $\{x_{\alpha, \beta}: 1 \mbox{ is an endpoint of } \alpha \}$.

\section{Future Directions} \label{sec:future}

 The associahedron has also been generalized to a type-$B$ associahedron, 
called the cyclohedron \cite{simion,bott-taubes} as well as other classes of polytopes such as graph associahedra \cite{carr-devadoss}, generalized associahedra \cite{fomin-zel-finite} and 
generalized permutahedra \cite{postnikov-permuto}. A natural extension of our work here is to use the same process to study some of these generalizations. 

We first look to the generalized associahedra of Fomin and Zelevinsky in \cite{fomin-zel-finite}. The generalized associahedra of type-$B$ has a description in terms of centrally symmetric 
triangulations. The type-$D$ associahedron also has a description in terms of centrally symmetric triangulations with some extra restrictions due to the similarities between the Coxeter groups 
of type-$B$ and $D$. In the case of the type-$B$ associahedron, there are 5-cycles appearing in the same way as the type-$A$ associahedron, and 6-cycles appearing when the central 
diagonal of a triangulation and an adjacent diagonal are removed. Using the same techniques we developed in this paper, it is easy to define and study exchange modules in 
the cluster algebra of type-$B_n$. The type-$D$ associahedron has a more difficult description that we will omit here, however it is easy to show that it contains only geodesic 4- 
and 5-cycles and that we can use the same techniques to study an exchange module for the type-$D_n$ cluster algebra. 

Within the realm of cluster algebras, there are combinatorial descriptions of the cluster complexes of any cluster algebra arising from a triangulated surface, as introduced by
 Fomin, Shapiro and Thurston \cite{fomin-shapiro-thurston}. A further, but possibly more difficult, extension of our work would be to study the discrete fundamental groups for this class of 
cluster complexes. 

We might also study discrete homotopy theory of the graph associahedra of Carr and Devadoss \cite{carr-devadoss}. 
This generalization gives combinatorial descriptions for polytopes in terms of connected components of graphs. Our initial investigations suggest that much 
of the work done in this paper regarding discrete homotopy theory will also carry over to these objects. 

Finally, we hope that the exchange module $E(\mathcal{A}_n)$ finds application in cluster algebras. An open problem is to generalize the definition of exchange module to other types of 
cluster algebras.

\section*{Acknowledgments}
We would like to thank Sergey Fomin for helpful discussions regarding cluster algebras. 
The second and third authors were partially supported by NSF grants DMS-0441170 and DMS-0932078,
administered by the Mathematical Sciences Research Institute while the
authors were in residence at MSRI during the Complementary Program, Fall 2009 - Spring 2011.
This work was developed during the
visit of the authors to MSRI and we thank the institute for its
hospitality.

\bibliographystyle{amsplain}
\bibliography{atbib}

\end{document}